\documentclass{amsart}
\usepackage{amsmath}

\newtheorem{theorem}{Theorem}[section] 
\newtheorem{lemma}[theorem]{Lemma}     

\newtheorem{proposition}[theorem]{Proposition}

\title[Probabilistic Method and Navier-Stokes]%
{The Probabilistic Method and large initial data for generalized Navier-Stokes systems}

\author{Jean C. Cortissoz}


\begin{document}

\begin{abstract}
In this paper we introduce a probabilistic approach to show the existence 
of initial data with arbitrarily large 
$L^2\left(\mathbb{R}^3\right)$,
$\dot{H}^{\frac{1}{2}}\left(\mathbb{R}^3\right)$ and
$\mathcal{PM}^2$-norms for which a Generalized Navier-Stokes system
generate a global regular solution. 
More precisely, we show that
from a certain family of possible large initial 
data most of them give raise to global
regular solutions to a given Generalized Navier-Stokes system.
\end{abstract}

\maketitle

\section{Introduction.}

The Navier-Stokes system in $\mathbb{R}^3$ can be written as,
\begin{equation}
\label{Navierstokes}
\left\{
\begin{array}{l}
u_t-\Delta u+u\cdot\nabla u+\nabla p=0 \quad \mbox{in} \quad \mathbb{R}^3\times \left(0,\infty\right)\\
\mbox{div}\, u=0 \\
u\left(x,0\right)=\psi\left(x\right)
\end{array}
\right.
\end{equation}
where $u=\left(u^1,u^2,u^3\right)$ 
is a vector field which represents the velocity of the fluid, and $p$ is a function that represents the pressure
of the fluid. In Fourier space, system (\ref{Navierstokes}) can be written as,
\begin{eqnarray}
\label{FourierNS0}
\hat{u}^l\left(\xi,t\right)&=&
\hat{\psi}^l\left(\xi\right)\exp\left(-\left|\xi\right|^2 t\right)\\ \notag
&&
+\int_{0}^{t}\exp\left(-\left|\xi\right|^2\left(t-s\right)\right)\int_{q\in\mathbb{R}^3}
M_{kjl}\left(\xi\right)\hat{u}^k\left(q,s\right)\hat{u}^j\left(\xi-q,s\right)\, ds.
\end{eqnarray}
and the divergence-free condition translates to
\begin{equation}
\label{divergencefree}
\xi_1\hat{u}^1\left(\xi_1,t\right)+\xi_2\hat{u}^2\left(\xi,t\right)+\xi_3\hat{u}^3\left(\xi,t\right)=0.
\end{equation}
Einstein summation convention is in use, and by this we mean that we are writing
\[
M_{kjl}\left(\xi\right)\hat{u}^k\left(q,s\right)\hat{u}^j\left(\xi-q,s\right)
:=\sum_{k,j=1,2,3} M_{kjl}\left(\xi\right)\hat{u}^k\left(q,s\right)\hat{u}^j\left(\xi-q,s\right).
\]

It is not difficult to show that when (\ref{FourierNS0}) represents the Navier Stokes system, 
\begin{equation}
\label{generalized}
\left|M_{kjl}\left(\xi\right)\right|\leq \left|\xi\right|.
\end{equation}
So, following Chemin and Gallagher in \cite{CG2}, we will say that equation (\ref{FourierNS0}) is a Generalized
Navier-Stokes system if it satisfies (\ref{generalized}).
Examples of Generalized Navier-Stokes systems are Montgomery-Smith's toy model 
for the Navier-Stokes equation (see \cite{MS}), and Gallagher and Paicu's  examples in \cite{GP}.

An important role 
in what follows is played by the pseudomeasure space $\mathcal{PM}^2$. This
space was defined by LeJan and
Sznitman in \cite{LJS} to study questions of global existence and uniqueness of
the 3D-Navier Stokes system, and were subsequently used by Cannone and Karch in \cite{CK}
to study singular solutions to the Navier-Stokes equations. $\mathcal{PM}^2$ is defined as  
\[
\mathcal{PM}^2=\left\{v\in\mathcal{S}'\left(\mathbb{R}^3\right):\,
\hat{v}\in L_{\mbox{loc}}^1\left(\mathbb{R}^3\right),\quad \left\|v\right\|_2
=\mbox{ess sup}_{\xi\in\mathbb{R}^3} \left|\xi\right|^2\left|\hat{v}\left(\xi\right)\right|<\infty
\right\}.
\] 
It is not difficult to show that $\mathcal{PM}^2$ with the norm thus defined is a Banach space.
Arnold and Sinai
define related spaces in the periodic case in \cite{Arnold}, again to study global
existence and uniqueness of
solutions to the Navier-Stokes equations. 
In general, it can be shown that if the initial condition $\psi$ is small in $\mathcal{PM}^2$, then
(\ref{FourierNS0}) has a global regular solution (see \cite{Arnold},\cite{CK}, \cite{LJS},\cite{Cortissoz2}).
By a solution to (\ref{FourierNS0}) we mean a function 
$u\left(t\right)=\left(u^1\left(t\right),u^2\left(t\right),u^3\left(t\right)\right)$,
\[
u: \left[0,T\right]\longrightarrow \left(\mathcal{PM}^2\right)^3
\]
with each component being weakly continuous (i.e., each $\hat{u}^i\left(\xi,t\right)$ is continuous in $t$), and
such that (\ref{FourierNS0}) holds. 

One fact that is appealing about the LeJan-Sznitman
spaces is that they allow the use of elementary tools to study the behavior of nonlinear parabolic systems 
(see \cite{Cortissoz3}). This is why,
inspired by the work of Chemin and Gallagher (\cite{CG1}, \cite{CG2}), and Chemin, Gallagher and Paicu
(\cite{CGP}) on the existence
of families of large initial conditions in certain homogeneous Besov spaces for which (\ref{Navierstokes}) have
global regular solutions, the initial motivation for writing this paper was to present a
method for proving
the existence of large initial conditions in $\mathcal{PM}^2$ whose $L^2\left(\mathbb{R}^3\right)$-norm
is also large and
for which (\ref{Navierstokes}) has global regular solutions, avoiding the trick of going to larger
spaces where small initial data-global regular solutions results hold. 
But, before we continue, let
us briefly discuss the
work of Chemin, Gallagher and Paicu. In \cite{CG1}, \cite{CG2} and \cite{CGP}, the authors construct families of initial data that have large
Besov $B_{\infty,\infty}^{-1}$-norm in the
case of periodic boundary conditions and large $\dot{B}^{-1}_{\infty,\infty}$ in the case of the whole space $\mathbb{R}^3$ for which
the Navier-Stokes equation have global regular solutions. The importance of the $B_{\infty,\infty}^{-1}$ and 
$\dot{B}_{\infty,\infty}^{-1}$
norms being large in these examples
is the fact that, since all
critical spaces for the Navier-Stokes equations (including $\mathcal{PM}^2$)
are continously
embedded in $B_{\infty,\infty}^{-1}\left(\mathbb{T}^3\right)$ and
$\dot{B}_{\infty,\infty}^{-1}\left(\mathbb{R}^3\right)$, this implies immediatly that there is no way to recover these results by
means of a theorem of the type ``small initial data implies global regular solution".

The question is then: granted
the results of Chemin, Gallagher and Paicu on the existence of large initial data that give raise to global regular solutions,
what about trying to show that actually in some sense, there is a plethora of such large initial conditions? 
Having this question in mind, the idea we follow in this paper goes along this line: 
we will not exhibit a family of functions for which a global regular solution exists,
but rather we will show that with a high probability in a given set there are 
such initial conditions. In this paper, we try as possible set of initial data a  
set of functions whose Fourier transform is supported in an annulus, and which should be thought
as a first approximation to the set of functions with bounded Fourier transform supported in a compact set -or
with nice decaying properties at infinity.  

Let us be more specific and state the main theorem of this paper. In order to do so,
we introduce some notation and a construction. Let 
\[
\mathcal{H}_{\kappa}=\left\{
\left(\xi_1,\xi_2,\xi_3\right):\, \left|\xi_3\right|\geq \kappa>0
\right\}, \quad \mathbb{H}_{\geq 0}
=\left\{\left(\xi_1,\xi_2,\xi_3\right):\, \xi_3\geq 0\right\},
\]
and
\[
R_{1,2}=
\left[-2,2\right)^3\setminus \left[-1,1\right)^3.
\]
Divide the set $R_{1,2}\cap \mathcal{H}_{\kappa}\cap \mathbb{H}_{\geq 0}$
into $K$ small disjoint congruent cubes of the form
$\left[h_1,h_2\right)\times\left[h_1',h_2'\right)\times\left[h_1'',h_2''\right)$, each of volume $O\left(\frac{1}{K}\right)$,
which
we shall call blocks; enumerate them as
$H_s$, $s=1,2,\dots, K$, and denote
the partition of $R_{1,2}$ into blocks by $\mathcal{P}$. 
Now, divide each block into $K^2$ disjoint congruent cubes, each of volume $O\left(\frac{1}{K^3}\right)$,
which we shall call subblocks; enumerate them as
$W_{s,p}$, $p=1,2,\dots, K^2$; furthermore, assume that $W_{s,p}\subset H_{s}$.
Extend this partition to $R_{1,2}\cap \mathcal{H}_{\kappa}$, by defining 
\[
\tilde{H}_{s}=-H_{s} \quad \mbox{and} \quad \tilde{W}_{s,p}=-W_{s,p}
\]
where, given a set $A$, the set $-A$ is defined by
\[
-A=\left\{-a\in \mathbb{R}^3:\, a\in A\right\}.
\]
Let $\left(\Gamma, \mathcal{G}, P\right)$ be a probability space, and 
\[
r_{s,p}^j:\Gamma \longrightarrow \left\{-1,1\right\},\quad j=1,2,
\]
random independent Bernoulli trials, i.e.,
\[
P\left[\gamma:\, r_{s,p}^j\left(\gamma\right)=1\right]=\frac{1}{2}=P\left[\gamma:\, r_{s,p}^j\left(\gamma\right)=-1\right].
\]
Fix complex numbers 
$\Theta^j_{K,s,p}$ such that $\left|\Theta^j_{K,s,p}\right|=\left(\log\log K\right)^{\frac{1}{4}}$.
Now, given $\gamma\in \Gamma$, define the function $\psi_{\gamma} \in \left(L^2\left(\mathbb{R}^3\right)\right)^3$ 
as follows. First, for $j=1,2$, define
\begin{equation}
\label{randomfunction}
\hat{\psi}_{\gamma}^j\left(\xi\right)=
\left\{
\begin{array}{l}
\mbox{(A)} \quad r_{s,p}^j\left(\gamma\right)\Theta^j_{K,s,p}/\left|\xi\right|^2 \quad \mbox{if} \quad \xi\in W_{s,p}\\
\\
\mbox{(B)} \quad \overline{\hat{\psi}_{\gamma}^j\left(\xi\right)} \quad \mbox{if} \quad -\xi \in W_{s,p}\\
\\
\mbox{(C)} \quad  0 \quad \mbox{otherwise}
\end{array}
\right.
\end{equation}
Recall that $\overline{z}$ represents the complex conjugate of $z$. Condition (B) is to ensure that $\psi$ is real valued.
And to ensure that the divergence-free condition is satisfied, define $\hat{\psi}^3\left(\xi\right)$ by
\[
\mbox{(D)} \quad
\hat{\psi}_{\gamma}^3\left(\xi\right)=-\frac{1}{\xi_3}\left(\xi_1\hat{\psi}_{\gamma}^1\left(\xi\right)+\xi_2\hat{\psi}_{\gamma}^2\left(\xi\right)\right)
\quad \mbox{if}\quad \xi\in R_{1,2}\cap \mathcal{H}_{\delta}
\]
and 
\[
\mbox{(E)}\quad\quad \hat{\psi}_{\gamma}^3\left(\xi\right)=0 \quad \mbox{otherwise}.
\]

Our main Theorem reads as follows.
\begin{theorem}
\label{maintheorem1}
For $K$ large enough, 
\[
P\left[\gamma:\, \psi_{\gamma}\,\, \mbox{generates a regular global solution to
(\ref{FourierNS0})}\right]\geq 1-\exp\left(-K\right).
\]
\end{theorem}

Notice that by definition, for $\xi\in R_{1,2}\cap \mathcal{H}_{\kappa}$, the $\psi_{\gamma}$'s satisfy an estimate
\[
C_1\left(\log\log K\right)^{\frac{1}{4}}\leq\left|\psi_{\gamma}^j\left(\xi\right)\right|
\leq C_2\left(\log\log K\right)^{\frac{1}{4}},\quad j=1,2,3,\quad C_1>0, 
\]
and hence, it is clear that
that as $K\rightarrow \infty$, the $L^2\left(\mathbb{R}^3\right)$, $\dot{H}^{\frac{1}{2}}\left(\mathbb{R}^3\right)$,
and $\mathcal{PM}^2$-norms 
of the elements of $\psi_{\gamma}$ go to $\infty$ as $\left(\log\log K\right)^{\frac{1}{4}}$,
so Theorem \ref{maintheorem1} implies the existence of large initial conditions in
the aforementioned spaces that give raise to global regular solutions to system (\ref{FourierNS0}). 
Theorem \ref{maintheorem1} can be related to a result of Cannone
 (``lemme remarquable": Lemma 3.3.8 in \cite{Cannone})
 which basically says that large initial data in 
$L^3\left(\mathbb{R}^3\right)$, and in particular 
in $\dot{H}^{\frac{1}{2}}\left(\mathbb{R}^3\right)$, that are highly oscillatory produce global regular solutions to the 
Navier-Stokes system: notice that a typical $\psi_{\gamma}$ has a strongly oscillating Fourier transform 
(this was pointed out by the referee).

Perhaps an important remark is in place here.
The choice of the set $R_{1,2}$ and the way of dividing it into congruent cubic blocks and subblocks, and the
choice of Bernoulli trials may seem quite
special to the reader. However, it can be seen from the proof of the main result, that we can go a little further in 
how general we can
make these choices. We discuss this issue briefly in the last section of this paper.

It is time to describe the two step strategy we will follow 
to prove Theorem \ref{maintheorem1}. In Section 2, we will show that for $\psi$ 
taken from a set of initial conditions if the nonlinear term
\begin{equation}
\label{nonlinearity}
\int_{q\in \mathbb{R}^3} M_{kjl}\left(\xi\right)\hat{\psi}^k\left(q\right)\hat{\psi}^j\left(\xi-q\right)\,d^3q
\end{equation}
is small then (\ref{FourierNS0}), with $\psi$ as initial condition, has a global
regular solution: this is the content of Section 2 
(similar ideas were used in \cite{CG1,CG2,CGP}). Observe 
that for the given $\psi_{\gamma}$'s, in a worst case scenario,
the nonlinear term is of order $\left(\log\log K\right)^{\frac{1}{2}}$, which 
is very large, so it is not obvious that there are initial conditions for 
which the nonlinear term is small.
Finally, in Section 3 we show that for the $\psi_{\gamma}$'s as defined above, for most $\xi\in R_{1,2}$, with 
high probability (\ref{nonlinearity}) is small. 
This method
of proving the existence of objects with certain properties should be reminiscent of Erd\"os' probabilistic method 
used in combinatorics, and hence the title of this paper. 
This method is also elementary in nature: no advanced knowledge
on Fourier or nonlinear analysis, or on functional inequalities is required. Also, and this is our hope,
this paper is a first step towards proving a theorem stating that ``generic initial data with
finite total kinetic energy generates a global regular solution to the Navier-Stokes equation".

The techniques proposed in this paper have been used with a slightly different flavor (see \cite{Deng} for
an application of similar methods to the Navier-Stokes equation, and \cite{Burq1}
and \cite{Burq2} for an application to the supercritical wave equation), 
to show that in certain spaces where no small initial data-
global regular solution or well-posedness results can be proved, for certain big subsets
of the space (in a probabilistic sense) existence of solutions does occur.

The author wants to express his gratitude to Professor M. Cannone for sending a copy of his book 
``Ondelettes, paraproduits et Navier-Stokes" upon request, to Jaime D\'avila and Guillermo
Rodr\'iguez-Blanco for discussing some aspects and results of this paper; to the referee for many
valuable comments that helped improving the exposition, specially of the probabilistic arguments, 
in this paper.

\section{Small nonlinear term implies global regular solution}

As announced in the introduction, the purpose of this section is to show that given $\psi\in \mathcal{I}_{K}$, if 
the nonlinear term
\[
\int_{q\in \mathbb{R}^3} M_{kjl}\left(\xi\right)\hat{\psi}^k\left(q\right)\hat{\psi}^j\left(\xi-q\right)\,d^3q
\]
is small enough, then (\ref{FourierNS0}) with initial condition $\psi$ has a global regular solution.

We have divided this section into four parts.
In the first part we introduce an scheme to produce solutions to (\ref{FourierNS0}); this method is inspired by a
delay device method used to prove existence of solutions to semilinear parabolic problems (see for instance Hamilton's
original proof of short time existence for the Ricci flow in \cite{Ham}).
In the second part, we present some important notation and definitions. 
Then, in the third part we
prove a few computational lemmas which are very useful in proving the
estimates in part four, where we finally reach the goal of the section -which is given in
its title.

\subsection{An iteration scheme.}

To study the existence and the behavior of solutions to (\ref{FourierNS0}) 
we use the following device which is very convenient
to our purposes. First we fix a time;
$T$, then we fix a small step size $\rho = \frac{1}{N}$ for $N$ very large, and define
$\tau_m=m\rho T$. We can construct solutions
to (\ref{FourierNS0}), by using the following scheme,
\begin{equation}
\label{FourierNS1}
\begin{array}{c}
\hat{u}_{\rho}^l\left(\xi,t\right)=
\hat{\psi}^l\left(\xi\right)\exp\left(-\left|\xi\right|^2 t\right)+\\ 
\int_{0}^{t-\rho T}\exp\left(-\left|\xi\right|^2\left(t-s\right)\right)\int_{q\in\mathbb{R}^3}
M_{kjl}\left(\xi\right)\hat{u}_{\rho}^i\left(q,s\right)\hat{u}_{\rho}^j\left(\xi-q,s\right)\,d^3q \,\, ds,
\end{array}
\end{equation}
and we use the following convention: if $t\leq 0$ then $\hat{u}^j\left(\xi,t\right)=\hat{\psi}^j\left(\xi\right)$.
To better see how this scheme can be used to find a solution to (\ref{FourierNS0}), we
use the equivalent formulation,
\begin{equation}
\label{FourierNS2}
\begin{array}{c}
\hat{u}_{\rho}^l\left(\xi,t\right)=
\hat{u}_{\rho}^l\left(\xi,\tau_n\right)\exp\left(-\left|\xi\right|^2 \left(t-\tau_n\right)\right)+\\ 
\int_{\tau_{n}}^{t-\rho T}\exp\left(-\left|\xi\right|^2\left(t-s\right)\right)\int_{q\in\mathbb{R}^3}
M_{kjl}\left(\xi\right)\hat{u}_{\rho}^k\left(q,s\right)
\hat{u}_{\rho}^j\left(\xi-q,s\right)\,d^3q\, ds,
\end{array}
\end{equation}
$t\in\left[\tau_n,\tau_{n+1}\right)$,
and we still are under the same convention: if $t\leq 0$ then $\hat{u}^j\left(\xi,t\right)=\hat{\psi}^j\left(\xi\right)$.
We hope it is now clear how to produce approximate solutions to (\ref{FourierNS0}) using
(\ref{FourierNS2}): once we have produced a solution to (\ref{FourierNS2}) on $\left[0,\tau_n\right)$,  we
use this information to extend the solution to $\left[\tau_n,\tau_{n+1}\right)$.
Also, this second formulation will allow us to have good control on $u^k_{\rho}$. 

\subsection{Important definitions, conventions and more notation.}
\label{notation}
To make our writing a bit less cumbersome, let us define a number $\omega$ such that
\[
\omega=\frac{1}{2^J}, \quad \mbox{where $J$ is a positive integer such that}\quad
\frac{1}{2}\frac{1}{K^{\frac{1}{8}}}\leq \omega <\frac{1}{K^{\frac{1}{8}}}.
\]
From now on we fix
$T=3\left(\log\log K\right)$, so $\tau_n=3\left(\log\log K\right)n\rho$, and let
\[
R_{m_1,m_2}=\left[-m_2,m_2\right)^3\setminus \left[-m_1,m_1\right)^3.
\]
Define a family of {\it good} sets
for $n=0,1,2,\dots,N$,
\[
\begin{array}{c}
\mathcal{E}_n=\\
\left\{
\xi\in R_{\omega,8}: 
\begin{array}{c}
\sum_{H\in \mathcal{P}} 
\left|\int_{q\in H}
M_{kjl}\left(\xi\right) \hat{u}^k_{\rho}\left(q,t\right)\hat{u}^j_{\rho}\left(\xi-q,t\right)\,d^3q\right|\leq E_{n}\\
\\
\quad \mbox{for all}\quad t\in\left[\tau_{n-1},\tau_n\right),\quad \mbox{and for all} \quad
k,j,l=1,2,3
\end{array}
\right\}
\end{array}
\]
where $E_0=\frac{1}{K^{\frac{1}{4}}}$.  
It also is worth noticing that
\[
\left|\int M_{kjl}\left(\xi\right) \hat{u}_{\rho}^k\left(q,t\right)\hat{u}_{\rho}^j\left(\xi-q,t\right)\right|\leq 
\sum_{H\in \mathcal{P}}\left|\int_{q\in H}M_{kjl}\left(\xi\right) \hat{u}_{\rho}^k\left(q,t\right)
\hat{u}_{\rho}^j\left(\xi-q,t\right)\,d^3 q\right|,
\]
and hence $\mathcal{E}_n$ is a set of frequencies for which the nonlinear term is small (of course as small
as $E_n$ dictates).
Indeed, even though 
the meaning of $E_n$ will be disclosed later, let us give an idea on what to expect: we will set
$E_0=\frac{1}{K^{\frac{1}{4}}}$ and then show
that $E_n\leq\frac {\exp\left(C\left(\log\log K\right)^{\frac{5}{2}}\right)}{K^{\frac{1}{14}}}$ holds for all $n$. 

Given the good set, we define the bad set as its complement,
i.e.,
\[
\mathcal{B}_n=R_{\omega,8}\setminus \mathcal{E}_n.
\] 
The reader will soon notice that $\mathcal{B}_n\subset \mathcal{B}_0$. The reason for this will be apparent 
from the proof of Lemma \ref{nonzerodensity}: $E_{n+1}$ is defined in terms of $E_n$, in such a way that
if $\xi \notin \mathcal{B}_{n}$ then $\xi \notin \mathcal{B}_{n+1}$.

Given a set $G\subset R_{\omega,8}$,
define a family of densities $\delta_j\left(G\right)$ as,
\[
\delta_{j}\left(G\right)=
\frac{\mu\left(G\cap R_{2^j\omega,2^{j+1}\omega}\right)}{\mu\left(R_{2^j\omega,2^{j+1}\omega}\right)},
\quad j=0,1,\dots J+3.
\]
where $\mu\left(A\right)$ represents the
Lebesgue measure of $A$.
In the case that $G=\mathcal{B}_n$ we employ the notation $\delta_{j,n}=\delta_j\left(\mathcal{B}_n\right)$.

Finally, we assume that the following bounds hold up to time $t=\tau_n$, $l=1,2,3$,
\begin{equation}
\label{bound1}
\left|\hat{u}_{\rho}^l\left(\xi,t\right)\right|\leq \frac{A_n}{\left|\xi\right|^{2}} 
\quad
\mbox{if} \quad\xi\in R_{1,2},\quad \xi\notin \mathcal{B}_n
\end{equation}

\begin{equation}
\label{bound2}
\left|\hat{u}^l_{\rho}\left(\xi,t\right)\right|
\leq \frac{B_n}{\left|\xi\right|^2} \quad \mbox{if} \quad \xi \in \mathcal{B}_n,
\end{equation}

\begin{equation}
\label{bound3}
\left|\hat{u}_{\rho}^l\left(\xi,t\right)\right|\leq \frac{a_n}{\left|\xi\right|^{2}} 
\quad \mbox{if}\quad \xi \in R_{0,\omega}, \,\xi\neq \left(0,0,0\right),
\end{equation}

\begin{equation}
\label{bound4}
\left|\hat{u}_{\rho}^l\left(\xi,t\right)\right|\leq \frac{b_n}{\left|\xi\right|^{2}} 
\quad \mbox{if} \quad \xi \in R_{\omega,1}\quad
\mbox{or} \quad \xi \in R_{2,\log\log K}, \quad \xi\notin \mathcal{B}_n.
\end{equation}

\begin{equation}
\label{bound5}
\left|\hat{u}_{\rho}^l\left(\xi,t\right)\right|\leq \frac{c_{n}}{\left|\xi\right|^{2}} \quad 
\mbox{if}\quad \xi \notin R_{0,\log \log K},
\end{equation}
and
\begin{equation}
\label{bound6}
\sum_{H\in \mathcal{P}} 
\left|\int_{q\in H}
M_{kjl}\left(\xi\right) \hat{u}_{\rho}^k\left(q,t\right)\hat{u}_{\rho}^j\left(\xi-q,t\right)\,d^3 q\right|
\leq E_{n}
\quad \mbox{if} \quad \xi \notin \mathcal{B}_n .
\end{equation}

\subsection{important computational lemmas}
In this section we present a sequence of lemmas that are helpful to estimate the nonlinear term.
From now on, by $h=O\left(g\right)$ it is meant that $h\leq C\cdot g$ for a constant
$C$ independent of $K$.
\begin{lemma}
\label{computation1}
Let $G\subset R_{\omega,8}$, and let $\delta_{j}\left(G\right)$ be its family of
densities. Assume that $\delta_{j} \leq \sigma$ for all $j$. Then there exists a universal
constant $c$ such that 
\[
\int_{q \in G}\frac{1}{\left|q\right|^2}\frac{1}{\left|\xi-q\right|^2}\,d^3 q\leq 
\frac{c\sigma^{\frac{1}{3}}}{\left|\xi\right|}
\]
\end{lemma}
\begin{proof}
First assume 
$\left|\xi\right|<\frac{\omega}{2}$. By the triangular inequality,
\[
\left|\xi-q\right|\geq \frac{\left|q\right|}{2}.
\]
Then,
\begin{eqnarray*}
\int_{q\in G}\frac{1}{\left|q\right|^2}\frac{1}{\left|\xi-q\right|^2}\,d^3 q
&\leq& 4 \int_{q \in G} \frac{1}{\left|q\right|^4}
=4\sum_{i}\int_{2^i\omega\leq \left|q\right|<2^{i+1}\omega}\frac{1}{\left|q\right|^4}\,d^3 q\\
&\leq& \sum_i\frac{4D\sigma \left(2^i \omega\right)^3}{\left(2^i\omega\right)^4}
\leq \frac{4D\sigma}{\omega}\leq \frac{4D\sigma}{\left|\xi\right|},
\end{eqnarray*}
where $D$ is a constant independent of $K$. From now on in this
proof $D$ will indicate a constant independent of $K$ that may change from line to line.

If $\left|\xi\right|\geq \frac{\omega}{2}$ we split 
\[
\int_{q\in G}\frac{1}{\left|q\right|^2}\frac{1}{\left|\xi-q\right|^2}\,d^3 q= I + II+III,
\]
and then compute
\begin{eqnarray*}
I= \int_{q\in G, 0<\left|q-\xi\right|<\frac{\left|\xi\right|}{2}}
\frac{1}{\left|q\right|^2}\frac{1}{\left|\xi-q\right|^2} \,d^3 q&\leq&
\frac{4}{\left|\xi\right|^2}\int_{q\in G, 0<\left|q-\xi\right|<\frac{\left|\xi\right|}{2}}
\frac{1}{\left|\xi-q\right|^2} \, d^3 q\\
&\leq& 
\frac{4D}{\left|\xi\right|^2}\sigma^{\frac{1}{3}}\left|\xi\right|=\frac{4D\sigma^{\frac{1}{3}}}{\left|\xi\right|}.
\end{eqnarray*}
To estimate the previous integral we took into consideration a worst case scenario: integral I is the largest 
possible if $G$ is contained in a ball centered at the origin whose volume is equal to the volume of $G$ wich
is $O\left(\sigma \omega^3\right)$. Therefore if $G$ is contained in such a ball, its radius would be 
$O\left(\sigma^{\frac{1}{3}}\omega\right)$, and the result follows.

Let $s\geq -1$ be such that 
$2^s\omega\leq \left|\xi\right|<2^{s+1}\omega$. Then we have,
\begin{eqnarray*}
&II=&\\
&\int_{q\in G, \omega\leq\left|q\right|<\frac{\left|\xi\right|}{2},
\frac{\left|\xi\right|}{2}\leq\left|\xi-q\right|<2\left|\xi\right|}\frac{1}{\left|q\right|^2}
\frac{1}{\left|\xi-q\right|^2}\,d^3q&\\
&\leq&\\
&\frac{4}{\left|\xi\right|^2}\int_{q\in G, \omega\leq\left|q\right|<\frac{\left|\xi\right|}{2}}
\frac{1}{\left|q\right|^2}\,d^3 q&\\
&\leq&\\
&\frac{4}{\left|\xi\right|^2}\sum_{i=0}^{s-1}\int_{2^i\omega\leq\left|q\right|<2^{i+1}\omega}
\frac{1}{\left|q\right|^2}\,d^3 q&\\
&\leq&\\
&\frac{4D}{\left|\xi\right|^2}\sum_{i=0}^{s-1}\frac{1}{2^{2i}\omega^2}\sigma 2^{3i}\omega^3&\\
&\leq&\\
&\frac{4D\sigma}{\left|\xi\right|^2}2\cdot 2^{s}\omega\leq \frac{4D\sigma}{\left|\xi\right|}.&
\end{eqnarray*}
Notice that if $s\leq 0$ then $II=0$.
Finally,
\begin{eqnarray*}
\int_{q\in G,\left|q\right|\geq\frac{\left|\xi\right|}{2},
\frac{\left|\xi\right|}{2}\leq\left|\xi-q\right|<2\left|\xi\right|}\frac{1}{\left|q\right|^2}
\frac{1}{\left|\xi-q\right|^2}\,d^3 q&\leq&
\int_{q\in G,\frac{\left|\xi\right|}{2}\leq q<\frac{5\left|\xi\right|}{2}}\frac{4}{\left|q\right|^4}\,d^3 q\leq 
\frac{D\sigma}{\left|\xi\right|},
\end{eqnarray*}
and since $\sigma\leq 1$, the Lemma follows.
\end{proof}

\begin{lemma}
\label{computation2}
Assume $\xi$ is such that $0\leq\left|\xi\right|<\frac{1}{K^{\frac{1}{8}}}$, then
there is a constant $D$ independent of $K$ such that
\[
\int_{\left|q\right|\geq 1}\frac{1}{\left|q\right|^2}\frac{1}{\left|\xi-q\right|^2}\, d^3 q
\leq \frac{D}{K^{\frac{1}{8}}}\frac{1}{\left|\xi\right|}.
\]
\end{lemma}
\begin{proof}
The triangular inequality implies that $\left|\xi-q\right|\geq \frac{\left|q\right|}{2}$
as long as $\left|q\right|\geq 1$ and $\left|\xi\right|<K^{-\frac{1}{8}}$ (this of
course for $K$ large enough). Hence,
\begin{eqnarray*}
\int_{\left|q\right|\geq 1}\frac{1}{\left|q\right|^2}\frac{1}{\left|\xi-q\right|^2}\,d^3 q\leq
4\int_{\left|q\right|\geq 1}\frac{1}{\left|q\right|^4}\,d^3 q\leq
D\leq \frac{D}{K^{\frac{1}{8}}\left|\xi\right|}.
\end{eqnarray*}
\end{proof}

\begin{lemma}
\label{computation3}
Assume $\xi$ is such that $\left|\xi\right|\geq \log \log K$, then there
is a constant $D$ independent of $K$ such that
\[
\int_{0\leq\left|q\right|<8}\frac{1}{\left|q\right|^2}\frac{1}{\left|\xi-q\right|^2}\,d^3 q
\leq \frac{D}{\log \log K}\frac{1}{\left|\xi\right|}.
\]
\end{lemma}
\begin{proof}
The triangular inequality implies that $\left|\xi-q\right|\geq \frac{\left|\xi\right|}{2}$
whenever $\left|\xi\right|\geq \log\log K$ and $\left|q\right|<4$. Using 
this we can bound,
\begin{eqnarray*}
\int_{0\leq\left|q\right|<8}\frac{1}{\left|q\right|^2}\frac{1}{\left|\xi-q\right|^2}\,d^3 q
\leq
\frac{4}{\left|\xi\right|^2}\int_{0\leq\left|q\right|<8}\frac{1}{\left|q\right|^2}\, d^3 q
\leq\frac{D}{\left|\xi\right|^2}\leq \frac{D}{\log \log K \left|\xi\right|}.
\end{eqnarray*}
\end{proof}

\subsection{Main estimates}
We shall show how to control, inductively,
each of the quantities $a_n,b_n,c_n, A_n, B_n$ and $E_n$ defined in Section \ref{notation}. 
But before we start once again, to make our writing easier, 
we shall introduce some more terminology. Fix an interval of time $\left[\tau_{n-1},\tau_n\right)$
(recall that $\tau_n=3\left(\log\log K\right)^{\frac{1}{4}}n\rho$),
if $\xi\notin \mathcal{B}_n$ is such that
\begin{itemize}
\item
$\left|\xi\right|<K^{-\frac{1}{8}}$, we call it a {\it low-low frequency}, and this set of frequencies is denoted by $lL$
\item
$K^{-\frac{1}{8}}\leq\left|\xi\right|<1$, we call it a {\it high-low frequency}, and use the notation $hL$ 
\item
$1\leq\left|\xi\right|<2$, we call it a {\it medium frequency}, and employ the notation $M$
\item
$2\leq\left|\xi\right|<\log \log K$, we call it a {\it low-high frequency}, and employ the notation $lH$
\item
$\left|\xi\right|\geq \log \log K$, we call it a {\it high-high frequency} and employ the notation $hH$.
\end{itemize}
The set of bad frequencies will be denoted by $B$ (i.e., we drop the dependence on $n$ once we have fixed an
interval of time). Let us give an example on how this notation will be used. If we are estimating
on the time interval $\left[\tau_{n-1},\tau_n\right)$, and
we write $\int_{B-lH}M_{kjl}\left(\xi\right) \hat{u}^k\left(q\right)\hat{u}^j\left(\xi-q\right)$ this actually means
\begin{eqnarray*}
&\int_{B-lH}M_{kjl}\left(\xi\right) \hat{u}^k\left(q\right)\hat{u}^j\left(\xi-q\right)&\\ 
&=& \\
&\left(\int_{q\in\mathcal{B}_n, \xi-q\in R_{2,\log\log K}}+
\int_{\xi-q\in\mathcal{B}_n, q\in R_{2,\log\log K}}\right) M_{kjl}\left(\xi\right) \hat{u}^k\left(q\right) \hat{u}^j\left(\xi-q\right);&
\end{eqnarray*}
and $\int_B$ is a shorthand for
\[
\int_{B-lL}+\int_{B-hL}+\int_{B-M}+\int_{B-B}+\int_{B-lL}+\int_{B-hH}.
\]
We are ready to state and prove
the following,
\begin{proposition}
\label{nonzerodensity}
Assume that bounds (\ref{bound1})-(\ref{bound6})
and $\delta_{j,n}\leq \sigma$ hold on $\left[\tau_{n-1},\tau_n\right)$, for all $j$. Then 
there exists a constant $\lambda$ independent of $\rho$, $n$ and $K$ 
such that if we define (here $\theta=\sigma^{\frac{1}{3}}$),
\begin{eqnarray*}
B_{n+1}&=&B_n+\lambda\left[a_n^2+a_nb_n+a_nA_n+b_nA_n+A_n^2+\right.\\
&&\left.
+b_nc_n+c_n^2+\theta B_n\left(a_n+b_n+c_n+A_n+B_n\right)\right]\left(\log\log K\right)^{\frac{1}{4}}\rho;
\end{eqnarray*}
\begin{eqnarray*}
a_{n+1}&=&a_n+\lambda\left[a_n^2+a_nb_n+\frac{1}{K^{\frac{1}{8}}}A_n^2
+\frac{1}{K^{\frac{1}{8}}}A_nb_n+\right.\\
&&\left.+\theta B_n\left(a_n+b_n+c_n+A_n +B_n\right)
+\frac{1}{K^{\frac{1}{8}}}c_n^2+\frac{1}{K^{\frac{1}{8}}}c_nb_n\right]
\left(\frac{\log\log K}{K}\right)^{\frac{1}{4}}\rho;
\end{eqnarray*}
\begin{eqnarray*}
b_{n+1}&=&b_n+\lambda\left[a_n^2+a_nb_n+a_nc_n+b_n^2+b_nc_n+A_nb_n+A_nc_n\right.\\
&&\left.+c_n^2+E_n+\theta B_n\left(a_n+b_n+c_n+A_n+B_n\right)\right]
\left(\log\log K\right)\left(\log\log K\right)^{\frac{1}{4}}\rho;
\end{eqnarray*}
\begin{eqnarray*}
E_{n+1}&=&E_n+\lambda\left(A_n+\theta B_n\right)\left[\frac{A_n}{K^{\frac{2}{7}}}+E_n
+ X_n\right]\left(\log\log K\right)^{\frac{1}{4}}\rho\\
&&+
\theta\lambda\left(A_n+B_n\right)\left(\frac{B_n}{K^{\frac{2}{7}}}+A_n^2+X_n\right)\left(\log\log K\right)^{\frac{1}{4}}\rho\\
&&
+\lambda^2 Y_n\left(\log\log K\right)^{\frac{1}{2}}\rho^2,
\end{eqnarray*}
where 
\begin{eqnarray*}
X_n&=&
\frac{1}{K^{\frac{1}{8}}}a_nb_n+\frac{1}{K^{\frac{1}{8}}}a_nA_n+A_nb_n+
b_n^2\\
&&+\frac{1}{\log\log K}b_nc_n
+\theta B_n\left(a_n+b_n+c_n+A_n+B_n\right)+\frac{1}{\log\log K}c_n^2,
\end{eqnarray*}
\begin{eqnarray*}
Y_n&=&\theta\left(\frac{B_n}{K^{\frac{2}{7}}}+A_n^2+X_n\right)^2\\
&&+ 2\theta
\left(\frac{B_n}{K^{\frac{2}{7}}}+A_n^2+X_n\right)\left(\frac{A_n}{K^{\frac{2}{7}}}+E_n+X_n\right)
+\left(\frac{A_n}{K^{\frac{2}{7}}}+E_n+X_n\right)^2
\end{eqnarray*}
and
\[
\begin{array}{c}
\mathcal{E}_{n+1}=\\
\left\{
\xi\in R_{\omega,8}: 
\begin{array}{c}
\sum_{C\in \mathcal{P}} 
\left|\int_{q\in C,\xi-q\in R_{1,2}}
M_{kjl}\left(\xi\right) \hat{u}_{\rho}^k\left(q,t\right)\hat{u}_{\rho}^j\left(\xi-q,t\right)\, d^3 q\right|
\leq E_{n+1}\\
t\in\left[\tau_n,\tau_{n+1}\right)
\end{array}
\right\},
\end{array}
\]
then bounds
(\ref{bound2}), (\ref{bound3}), (\ref{bound4}), (\ref{bound6}) hold on $\left[\tau_{n},\tau_{n+1}\right)$.
\end{proposition}

\begin{proof}
Let us show how to obtain the expression for $a_{n+1}$. To this end,
let $\xi$ be a low-low frequency ($lL$). We are going to estimate $u_{\xi}$ on $\left[\tau_n,\tau_{n+1}\right)$.
In order to do this, we decompose the nonlinear term 
$\int M_{kjl}\left(\xi\right) \hat{u}^k\left(q\right)\hat{u}^j\left(\xi-q\right)$ 
into sums of interactions between the 
possible different frequencies.
Of course,
we only take into account those interactions that
can appear in the nonlinear term for a low-low frequency.
For instance, interactions such as $lL-M$, $hL-lH$, or $M-hH$ are precluded by the triangular inequality.
Hence, by considering only the possible interactions, using the assumed bounds 
(\ref{bound1}) and (\ref{bound2}), and recalling that if $t\in\left[\tau_n,\tau_{n+1}\right)$
then $t-\rho T\in \left[\tau_{n-1},\tau_{n}\right)$, with the help of 
Lemmas \ref{computation1}, \ref{computation2} and \ref{computation3}, we get
(as before, in what follows $D$ represents a generic universal constant that may change from line to line),
\[
\left|\int_{lL-lL}M_{kjl}\left(\xi\right) \hat{u}_{\rho}^k\left(q,t\right)\hat{u}_{\rho}^j\left(\xi-q,t\right)\right|\leq D a_n^2,
\]
\[
\left|\int_{lL-hL}M_{kjl}\left(\xi\right) \hat{u}_{\rho}^k\left(q,t\right)\hat{u}_{\rho}^j\left(\xi-q,t\right)\right| \leq 2D a_nb_n,
\]
\[
\left|\int_{hL-hL}M_{kjl}\left(\xi\right)\hat{u}_{\rho}^k\left(q,t\right)\hat{u}_{\rho}^j\left(\xi-q,t\right)\right|\leq D b_n^2,
\]
\[
\left|\int_{M-M} M_{kjl}\left(\xi\right) \hat{u}_{\rho}^j\left(q,t\right)\hat{u}_{\rho}^k\left(\xi-q,t\right)\right|\leq \frac{D}{K^{\frac{1}{8}}}A_n^2,
\]
\[
\left|\int_{M-lH} M_{kjl}\left(\xi\right) \hat{u}_{\rho}^k\left(q,t\right)\hat{u}_{\rho}^j\left(\xi-q,t\right)\right|\leq \frac{2D}{K^{\frac{1}{8}}}A_nb_n,
\]
\[
\left|\int_{hL-M}M_{kjl}\left(\xi\right) \hat{u}_{\rho}^k\left(q,t\right)\hat{u}_{\rho}^j\left(\xi-q,t\right)\right|\leq \frac{2D}{K^{\frac{1}{8}}}A_nb_n,
\]
\[
\left|\int_{lH-hH}M_{kjl}\left(\xi\right) \hat{u}_{\rho}^k\left(q,t\right)\hat{u}_{\rho}^j\left(\xi-q,t\right)\right|\leq \frac{2D}{K^{\frac{1}{8}}}b_nc_n,
\]
\[
\left|\int_{hH-hH} M_{kjl}\left(\xi\right)\hat{u}_{\rho}^k\left(q,t\right)\hat{u}_{\rho}^j\left(\xi-q,t\right)\right|\leq \frac{D}{K^{\frac{1}{8}}}c_n^2;
\]
and for interactions involving frequencies in the bad set we get
\[
\left|\int_{B}M_{jkl}\left(\xi\right) \hat{u}_{\rho}^j\left(q,t\right)\hat{u}_{\rho}^k\left(\xi-q,t\right)\right|\leq D\theta
B_n\left(a_n+b_n+c_n+d_n+A_n+B_n\right).
\]
Collecting all the previous estimates, we obtain the following bound for 
the nonlinear term, when $\xi$ is an $lL$ frequency,
\[
\begin{array}{c}
\left|\int M_{kjl}\left(\xi\right) \hat{u}_{\rho}^k\left(q,t\right)\hat{u}_{\rho}^j\left(\xi-q,t\right)\right|\leq
\lambda\left[a_n^2+a_nb_n+\frac{1}{K^{\frac{1}{8}}}A_n^2 \right.\\
\left.+\frac{1}{K^{\frac{1}{8}}}A_nb_n
+\theta B_n\left(a_n+b_n+c_n+A_n +B_n\right)
+\frac{1}{K^{\frac{1}{8}}}c_n^2+\frac{1}{K^{\frac{1}{8}}}c_nb_n\right],
\end{array}
\]
where $\lambda$ is a universal constant.
Now, we plug this into (\ref{FourierNS2}), integrate and use the estimate,
\[
1-\exp\left(-\left|\xi\right|^2\left(t-\tau_n\right)\right)\leq 3\left(\frac{\log\log K}{K}\right)^{\frac{1}{4}}\rho,
\]
which holds as long as  $t\in\left[\tau_n,\tau_{n+1}\right)$ and $\left|\xi\right|< K^{-\frac{1}{8}}$, to 
obtain the following estimate for $\xi$ an $lL$ frequency, 
\[
\begin{array}{c}
\left|\xi\right|^2 \left|\hat{u}_{\rho}^{k}\left(\xi,t\right)\right|\leq
a_n+\lambda\left[a_n^2+a_nb_n+\frac{1}{K^{\frac{1}{8}}}A_n^2
+\frac{1}{K^{\frac{1}{8}}}A_nb_n+\qquad\qquad\qquad\qquad\quad\right.\\
\quad\left.+\theta B_n\left(a_n+b_n+c_n+A_n +B_n\right)
+\frac{1}{K^{\frac{1}{8}}}c_n^2+\frac{1}{K^{\frac{1}{8}}}c_nb_n\right]
\left(\frac{\log\log K}{K}\right)^{\frac{1}{4}}\rho,
\end{array}
\]
with a universal constant $\lambda$.
Hence by defining $a_{n+1}$ as the lefthandside of the previous inequality, it follows that 
(\ref{bound3}) 
holds up to time $\tau_{n+1}$. The expressions for $b_n$ and $c_n$ can be obtained in a similar fashion. 
To compute the expression for $b_{n+1}$, it must be taken into account that
there are some $lH$ frequencies $\xi$ which satisfy $\left|\xi\right|=O\left(\log\log K\right)$, and this is the
reason for the extra $\log\log K$ in the expression for $b_{n+1}$.

Let us also sketch how to obtain the expression for $E_{n+1}$.
If $\alpha\in \mathcal{B}_n$, then we can bound the nonlinear term by
\begin{eqnarray*}
&\left|\int M_{kjl}\left(\alpha\right) \hat{u}_{\rho}^k\left(q,t\right)\hat{u}_{\rho}^j\left(\alpha-q,t\right)\right|&\\
&\leq&\\
 &D\left(\frac{1}{K^{\frac{1}{8}}}a_nb_n+\frac{1}{K^{\frac{1}{8}}}a_nA_n+
b_n^2+\frac{1}{\log\log K}b_nc_n+\right.&\\
&\left.+A_n^2+\theta B_n\left(a_n+b_n+c_n+A_n+B_n\right)+\frac{1}{\log\log K}c_n^2\right)&\\
&=&\\
&D\left(A_n^2+X_n\right).&
\end{eqnarray*}
Therefore, if $\alpha\in \mathcal{B}_n$, for $t\in\left[\tau_n,\tau_{n+1}\right)$, we obtain
\[
u_{\alpha}\left(t\right)=\exp\left(-\left|\alpha\right|^2\left(t-\tau_{n}\right)\right)u_{\xi}\left(\tau_{n}\right)
+\mbox{error term 1}.
\]
and this error term satisfies,
\[
\left|\mbox{error term 1}\right|\leq D
\frac{1}{\left|\alpha\right|^2}\left(A_n^2+X_n\right)\left(\log\log K\right)^{\frac{1}{4}}\rho.
\]
Now notice that if $q$ and $q'$ belong
to the same block, we have that 
\[
\left|q-q'\right|\leq (const)K^{-\frac{1}{3}},
\] 
and hence,
since $t-\tau_n\leq 3\left(\log\log K\right)^{\frac{1}{4}}$, 
\[
\exp\left(-\left|q\right|^2\left(t-\tau_n\right)\right)
=\exp\left(-\left|q'\right|^2\left(t-\tau_n\right)\right)
+O\left(K^{-\frac{1}{3}}\left(\log\log K\right)^{\frac{1}{4}}\rho\right).
\]
Let $H$ be the block to which $\alpha$ belongs and fix any $q_H\in H$. Let 
$$\eta_H\left(t\right)=\exp\left(-\left|q_H\right|^2\left(t-\tau_n\right)\right).$$
Then by our previous remark, 
\[
\hat{u}^j_{\rho}\left(\alpha,t\right)=\eta_H\left(t\right) \hat{u}^j_{\rho}\left(\alpha,\tau_{n}\right)
+\mbox{error term 2},
\]
and this error term satisfies
\begin{eqnarray*}
\left|\mbox{error term 2}\right|\leq
\frac{D}{\left|\alpha\right|^2}\left[B_nO\left(K^{-\frac{1}{3}}\right)
+A_n^2+X_n\right]\left(\log\log K\right)^{\frac{1}{4}}\rho.
\end{eqnarray*}
So for $K$ large enough we obtain the bound,
\[
\left|\mbox{error term 2}\right|\leq
\frac{D}{\left|\alpha\right|^2}\left(\frac{B_n}{K^{\frac{2}{7}}}+
A_n^2+X_n\right)\left(\log\log K\right)^{\frac{1}{4}}\rho.
\]
On the other hand, if $\alpha\in \mathcal{E}_n$, we can bound the nonlinear term as follows
\[
\begin{array}{c}
\left|\int_{q\in\mathbb{R}^3} M_{kjl}\left(\alpha\right) \hat{u}_{\rho}^k\left(q,t\right)\hat{u}_{\rho}^j\left(\alpha-q,t\right)\right|\leq
\lambda\left(\frac{1}{K^{\frac{1}{8}}}a_nb_n+\frac{1}{K^{\frac{1}{8}}}a_nA_n+A_nb_n+\right.\\
\left.\theta B_n\left(a_n+b_n+c_n+A_n+B_n\right)+\frac{1}{\log\log K}c_nb_n+
\frac{1}{\log\log K}c_n^2\right)+E_n.
\end{array}
\]
Again, as in the case of a bad frequency we have
\[
\hat{u}^j_{\rho}\left(\alpha,t\right)=\eta_H\left(t\right) \hat{u}^j_{\rho}\left(\xi,\tau_{n}\right)+
\mbox{error term 3},\quad t\in\left[\tau_n,\tau_{n+1}\right),
\]
and we can bound
\begin{eqnarray*}
\left|\mbox{error term 3}\right|\leq
\frac{1}{\left|\alpha\right|^2}\left(A_nO\left(K^{-\frac{1}{3}}\right)
+E_n+X_n\right)\left(\log\log K\right)^{\frac{1}{4}}\rho,
\end{eqnarray*}
to obtain for $K$ large enough the estimate
\[
\left|\mbox{error term 3}\right|\leq 
\frac{1}{\left|\alpha\right|^2}\left(\frac{A_n}{K^{\frac{2}{7}}}+
E_n+X_n\right)\left(\log\log K\right)^{\frac{1}{4}}\rho.
\]

For the next few lines keep in mind that $\left|q-q'\right|\leq (const) K^{-\frac{1}{3}}$, 
and hence that
\[
\left|\left(\xi-q\right)-\left(\xi-q'\right)\right|\leq (const) K^{-\frac{1}{3}}.
\]
Let 
$$\eta_{H,\xi}\left(t\right)=\exp\left(-\left|\xi-q_H\right|^2\left(t-\tau_n\right)\right),$$ 
with $q_H$ as previously defined.
Then, from the previous estimates, taking into account that $\alpha$
and $\xi-\alpha$ can be either a good or a bad frequency
and that $q$ and $q'$ belong
to the same block, it follows that the nonlinear term
satisfies the inequality
\[
\begin{array}{c}
\sum_{H\in\mathcal{P}}\left|\int_{\alpha\in H,\xi-\alpha\in R_{1,2}} 
M_{kjl}\left(\xi\right) \hat{u}_{\rho}^k\left(\alpha,t\right)\hat{u}_{\rho}^j\left(\xi-\alpha,t\right)\right|\\
\leq\\
\sum_{H\in \mathcal{P}}
\eta_H\left(t\right)\eta_{H,\xi}\left(t\right)
\left|\int_{\alpha\in H,\xi-\alpha \in R_{1,2}}
M_{kjl}\left(\xi\right) \hat{u}_{\rho}^k\left(\alpha,\tau_n\right)\hat{u}_{\rho}^j\left(\xi-\alpha,\tau_n\right)\right|
+\mbox{terms},
\end{array}
\]
where 
\[
\begin{array}{rl}
\mbox{terms}&=\\
&\lambda\left(A_n+\theta B_n\right)\left[\frac{A_n}{K^{\frac{2}{7}}}+E_n
+ X_n\right]\left(\log\log K\right)^{\frac{1}{4}}\rho+\\
&+
\theta\lambda\left(A_n+B_n\right)\left(\frac{B_n}{K^{\frac{2}{7}}}+A_n^2+X_n\right)\left(\log\log K\right)^{\frac{1}{4}}\rho+\\
&
+\lambda^2 Y_n\left(\log\log K\right)^{\frac{1}{2}}\rho^2.
\end{array}
\]
Since $\eta_H,\eta_{H,\xi}\leq 1$ for any $H$, if $\xi\in \mathcal{E}_{n}$, using (\ref{bound6}) we obtain
\[
\sum_{H\in\mathcal{P}}\left|\int_{\alpha\in H,\xi-q\in R_{1,2}} 
M_{kjl}\left(\xi\right) \hat{u}_{\rho}^k\left(q,t\right)\hat{u}_{\rho}^j\left(\xi-q,t\right)\right|\leq E_n+\mbox{other terms}
\]
and hence by defining $E_{n+1}$ as the righthandside of the previous inequality and $\mathcal{E}_{n+1}$ 
defined as in the statement,
we see that 
bound (\ref{bound6}) holds on $\left[\tau_{n},\tau_{n+1}\right)$, and also that 
$\mathcal{E}_{n}\subset \mathcal{E}_{n+1}$ 
(or which is the same, that $\mathcal{B}_{n+1}\subset \mathcal{B}_{n}$).
\end{proof}

Now we use expressions given by the previous proposition to
provide uniform bounds on $A_n, B_n, a_n, b_n$ and $c_n$ when $\rho\rightarrow 0$ (or which is the same, in $n$),
given some assumptions on $A_0, B_0, a_0, b_0, c_0$ and on a bound on $\delta_{j,n}$ (the reader should have
already noticed that $\mathcal{B}_{n}\subset \mathcal{B}_0$, and hence all is needed 
to obtain a uniform bound on $\delta_{j,n}$ is a bound on the
family of densities of $\mathcal{B}_0$). 

\begin{lemma}
\label{importantbounds}
Assume that $A_0\leq M\left(\log\log K\right)^{\frac{1}{4}}$, $B_0\leq M\left(\log\log K\right)^{\frac{1}{4}}$ ($M\geq 1$),
$E_0\leq \frac{1}{K^{\frac{1}{4}}}$, $\sigma \leq \frac{1}{K^{\frac{1}{4}}}$
and $a_0=b_0=c_0=0$. 
Then for $K$ large enough the following estimates hold as long as $n\rho\leq 1$,
\[
a_n\leq \frac{1}{K^{\frac{1}{16}}}\left(1+\rho\right)^{n-1}n\rho; 
\]
\[
b_n\leq\frac{1}{K^{\frac{1}{14}}}\left[1+12M^2\left(\log\log K\right)^{\frac{5}{2}}\rho\right]^{n-1}n\rho;
\]
\[
c_n\leq \frac{1}{K^{\frac{1}{17}}};
\]
\[
A_n\leq 2M\left(\log\log K\right)^{\frac{1}{4}};
\]
\[
B_n\leq M\left(\log\log K\right)^{\frac{1}{4}}+\left[1+4M^2\left(\log\log K\right)^{\frac{5}{2}}\right]
\left(\log\log K\right)^{\frac{1}{4}}n\rho;
\]
and
\[
E_n\leq \frac{1}{K^{\frac{1}{14}}}\left(1+12M^2\left(\log\log K\right)^{\frac{5}{2}}\rho\right)^n n\rho.
\]
\end{lemma}

\begin{proof}
We show that the estimates are true by induction. 
We begin by showing that the bound on $a_n$ holds. The case $n=0$ is obvious.
To show the inductive step, let us rewrite the expresion for $a_{n+1}$ as
\begin{eqnarray*}
a_{n+1}= a_n\left[1+\lambda\left(a_n+b_n+\theta B_n\right)\left(\log\log K\right)^{\frac{1}{4}}\rho\right]+R_n,
\end{eqnarray*}
where 
\[
R_n=\lambda\left[\theta B_n\left(b_n+c_n+A_n+B_n\right)+\frac{1}{K^{\frac{1}{8}}}\left(c_n^2+c_nb_n\right)\right]
\left(\log\log K\right)^{\frac{1}{4}}\rho.
\]
Using the bounds on $\theta, a_n, b_n, c_n, A_n$, it is easy to show that
\[
\lambda\left(a_n+b_n+\theta B_n\right)\left(\log\log K\right)^{\frac{1}{4}}\leq 1,
\]
and
\[
R_n\leq \frac{1}{K^{\frac{1}{14}}}\left(\log\log K\right)^{\frac{1}{4}}\rho\leq \frac{1}{K^{\frac{1}{16}}}\rho.
\]
Putting all this information together yields,
\begin{eqnarray*}
a_{n+1}&\leq& a_n\left(1+\rho\right)+\frac{1}{K^{\frac{1}{16}}}\rho\\
&\leq& \frac{1}{K^{\frac{1}{16}}}\left(1+\rho\right)^{n-2}\left(n-1\right)\rho\cdot 
\left(1+\rho\right)+\frac{1}{K^{\frac{1}{16}}}\rho\\
&\leq& \frac{1}{K^{\frac{1}{16}}}\left(1+\rho\right)^{n-1}n\rho.
\end{eqnarray*}
 
Let us show the bound on $c_n$. Notice that if $\xi$ is a high-high frequency, the nonlinear term
can be bounded by
\begin{eqnarray*}
\lambda\left[b_n^2+b_nc_n+\left(b_n+\frac{c_n}{\log \log K}\right)\left(A_n+\theta B_n+a_n\right)
+c_n^2\right]\leq \frac {1}{2K^{\frac{1}{17}}};
\end{eqnarray*}
hence the real part of $\hat{u}_{\rho}^l\left(\xi,t\right)$ satisfies an equation
\[
\frac{d}{dt}Re\left(\hat{u}_\rho^l\left(\xi,t\right)\right)\leq 
-\left|\xi\right|^2 Re\left(\hat{u}_{\rho}^l\left(\xi,t\right)\right)+
\frac{1}{2K^{\frac{1}{17}}}, 
\]
which shows that $\left|Re\left(u_{\rho}^l\left(\xi,t\right)\right)\right|$ 
remains smaller than $\frac{1}{K^{\frac{1}{17}}}$.
Proceeding in the same way for the imaginary part, the result follows. 

To show the bound on $A_n$, recall that $A_0\leq M\left(\log\log K\right)^{\frac{1}{4}}$, and hence
\[
\left|Re\left(\hat{u}_{\rho}^l\left(\xi,0\right)\right)\leq M\left(\log\log K\right)^{\frac{1}{4}}\right|\quad
\mbox{and} \quad \left|Im\left(\hat{u}_{\rho}^l\left(\xi,0\right)\right)\leq M\left(\log\log K\right)^{\frac{1}{4}}\right|.
\]
On the other hand, If $\xi$ is a good frequency (i.e. not in the bad set), with $\xi\in R_{1,2}$, the nonlinear term can
be bounded by
\[
E_n+X_n\leq \frac{1}{K^{\frac{1}{16}}}.
\]
Therefore,
\[
\frac{d}{dt}Re\left(\hat{u}_{\rho}^l\left(\xi,t\right)\right)\leq 
-\left|\xi\right|^2 Re\left(\hat{u}_{\rho}^l\left(\xi,t\right)\right)+\frac{1}{K^{\frac{1}{16}}}.
\]
This shows that if $Re(\hat{u}_{\rho}^l\left(\xi,t\right))$ is close to $M\left(\log\log K\right)^{\frac{1}{4}}$, then
it is decreasing. Proceeding in the same way with the imaginary part of $\hat{u}_{\rho}\left(\xi,t\right)$, 
we conclude that
$\left|\hat{u}_{\rho}^l\left(\xi,t\right)\right|\leq 2 M\left(\log\log K\right)^{\frac{1}{4}}$.

Now we work on the bound on $B_n$. Notice that by the estimates we already have, it is easy to show that for
$K$ large,
\[
\begin{array}{c}
\lambda\left[a_n^2+a_nb_n+a_nA_n+b_nA_n+A_n^2 +b_nc_n+c_n^2 \right.\\
\left.+\theta B_n\left(a_n+b_n+c_n+A_n+B_n\right)\right]\\
\leq 1+4\left(\log\log K\right)^{\frac{1}{2}},
\end{array}
\]
and then we have that
\begin{eqnarray*}
B_{n+1}&\leq& B_n+\left(1+4M^2\left(\log\log K\right)^{\frac{1}{2}}\right)\left(\log\log K\right)^{\frac{1}{4}}\rho\\
&=&M\left(\log\log K\right)^{\frac{1}{4}}+
\left(1+4M^2\left(\log\log K\right)^{\frac{1}{2}}\right)\left(\log\log K\right)^{\frac{1}{4}}\left(n+1\right)\rho,
\end{eqnarray*}
which shows the estimate.

Bounding $E_n$ (and $b_n$) is a bit more subtle than the previous cases, since
$E_{n+1}$ depends on $b_n$ (and $b_{n+1}$ depends on $E_n$). To bound $E_n$, first we write,
\begin{eqnarray*}
E_{n+1}&=&E_n\left[1+\lambda\left(A_n+\theta B_n\right)\left(\log\log K\right)^{\frac{1}{4}}\rho\right]
+Q_n,
\end{eqnarray*}
where
\[
\begin{array}{c}
Q_n=\lambda\left(A_n+\theta B_n\right)\left[\frac{A_n}{K^{\frac{2}{7}}}
+ X_n\right]\left(\log\log K\right)^{\frac{1}{4}}\rho\\
+
\theta\lambda\left(A_n+B_n\right)\left(\frac{B_n}{K^{\frac{2}{7}}}+A_n^2+X_n\right)\left(\log\log K\right)^{\frac{1}{4}}\rho
+\lambda^2 Y_n\left(\log\log K\right)^{\frac{1}{2}}\rho^2.
\end{array}
\]
It is not difficult to produce the estimates
\[
\frac{A_n}{K^{\frac{2}{7}}}\leq \frac{1}{K^{\frac{1}{10}}},\quad
\frac{B_n}{K^{\frac{2}{7}}}\leq \frac{1}{K^{\frac{1}{10}}},\quad
\theta A_n\leq \frac{1}{K^\frac{1}{13}},\quad \theta B_n \leq \frac{1}{K^{\frac{1}{13}}},
\]
\[
X_n \leq \frac{1}{4K^{\frac{1}{13}}}+A_nb_n.
\]
Hence,
\begin{eqnarray*}
\lambda\left(A_n+\theta B_n\right)\left[\frac{A_n}{K^{\frac{2}{7}}}+X_n\right]
&\leq&
\left(2M\left(\log\log K\right)^{\frac{1}{4}}+1\right)\cdot\left(\frac{1}{K^{\frac{1}{13}}}+A_nb_n\right).
\end{eqnarray*}
By choosing $\rho\ll \frac{1}{\log\log K}$, it
is not difficult to get
\[
\lambda^2 Y_n \left(\log\log K\right)^{\frac{1}{2}}\rho^2
\leq \frac{1}{K^{\frac{1}{14}}}\rho.
\]
From this we obtain
\begin{eqnarray*}
E_{n+1}&\leq& E_n\left[1+\left(2M\left(\log\log K\right)^{\frac{1}{4}}+2\right)
\left(\log\log K\right)^{\frac{1}{4}}\rho\right]+\\
&&
\left(2M\left(\log\log K\right)^{\frac{1}{4}}+1\right)\left(2M\left(\log\log K\right)^{\frac{1}{4}}\right)b_n
\left(\log\log K\right)^{\frac{1}{4}}\rho
+
\frac{1}{K^{\frac{1}{14}}}\rho.
\end{eqnarray*}
If $E_n\leq b_n$, then we obtain the estimate
\begin{eqnarray*}
E_{n+1}&\leq& b_n\left[1+12M^2\left(\log\log K\right)^{\frac{1}{2}}\right]+\frac{1}{K^{\frac{1}{14}}}\rho\\
&\leq& \frac{1}{K^{\frac{1}{14}}}\left[1+12M^2\left(\log\log K\right)^{\frac{5}{2}}\right]^{n}\left(n\rho\right)
\left[1+12 M^2\left(\log\log K\right)^{\frac{1}{2}}\right]+\frac{1}{K^{\frac{1}{14}}}\rho,
\end{eqnarray*}
and the required estimate on $E_n$ follows. On the other hand, the same can be deduced if $b_n\leq E_n$. Finally, one can
apply a completely analogous argument to prove the estimate on $b_n$.
\end{proof}

To show that the scheme proposed converges towards a solution to the 
Generalized Navier-Stokes system, the reader should notice that 
for $\xi\in\left(0,\omega\right)^3$, $\hat{u}_{\rho}\left(\xi,t\right)$ remains uniformly
bounded on $\left(0,3\left(\log\log K\right)^{\frac{1}{4}}\right)$ independent of $\rho$. In order to see
this, define $a_{n,i}$ such that 
\begin{equation}
\label{finerbound}
\left|\hat{u}_{\rho}^l\left(\xi,t\right)\right|\leq \frac{a_{n,i}}{\left|\xi\right|^{2}} 
\quad \mbox{if}\quad \xi \in R_{\frac{1}{2^{i+1}}\omega,\frac{1}{2^i}\omega},\quad i=0,1,2,\dots,
\,\,\,t\in\left[\tau_{n-1},\tau_n\right).
\end{equation}
Observe that whenever $\xi\in \left(0,\omega\right)^3$, $\left|\xi\right|^2 t$ is quite small
($t\in\left(0,3\left(\log\log K\right)^{\frac{1}{4}}\right)$), and hence,
proceeding as in the proof of Proposition \ref{nonzerodensity} we have that if we define
\[
a_{n+1,i}=a_{n,i}+V_n \left(\frac{\omega}{2^i}\right)^2\left(\log\log K\right)^{\frac{1}{4}}n\rho,
\]
with $V_n$ given by
\[
\begin{array}{l}
\lambda\left[a_n^2+a_nb_n+\frac{1}{K^{\frac{1}{8}}}A_n^2
+\frac{1}{K^{\frac{1}{8}}}A_nb_n \right.\\
\qquad\qquad\left. +\theta B_n\left(a_n+b_n+c_n+A_n +B_n\right)
+\frac{1}{K^{\frac{1}{8}}}c_n^2+\frac{1}{K^{\frac{1}{8}}}c_nb_n\right],
\end{array}
\]
where $\lambda>0$ is a constant independent of $n$ and $K$, that bound (\ref{finerbound}) holds with
$a_{n,i}$ replaced by $a_{n+1,i}$ on $\left[\tau_n,\tau_{n+1}\right)$ (compare this statement with Proposition 
\ref{nonzerodensity}). 
From this, by an inductive argument, it can be shown then
that $$a_{n,i}=O\left(\left(\frac{\omega}{2^i}\right)^2\left(\log\log K\right)^{\frac{1}{4}}\right).$$ Recalling
the definition of $a_{n,i}$ our claim follows. From the techniques in \cite{Cortissoz} 
(see the section on regularity in \cite{Cortissoz2}), it follows that,
$\hat{u}_{\rho}\left(\xi,t\right)$ decays faster than any polynomial, and 
that the rate of decay is independent of $\rho$.
This fact, together with our previous observation shows that for any $t_0>0$, the
Sobolev 
$H^{\zeta}\left(\mathbb{R}^3\right)$-norms of the 
elements of the sequence of $\hat{u}_{\rho}\left(\xi,t\right)$ are uniformly bounded, independently of $\rho$. A
diagonal procedure together with a compactness argument then shows that the scheme converges towards a solution
of the Generalized Navier-Stokes system (\ref{FourierNS0}).

Notice then that the solution thus obtained is smooth in 
$\left(0,3\left(\log\log K\right)^{\frac{1}{4}}\right)$, since for $\left|\xi\right|\geq \log\log K$ 
by the bounds proved for the approximate solutions $u_{\rho}$, 
$u$ satisfies
\[
\sup \left|\xi\right|^2 \left|\hat{u}^l\left({\xi},t\right)\right|<\frac{1}{K^{\frac{1}{16}}} 
\quad \mbox{for} \quad t\in \left(0,3\left(\log\log K\right)^{\frac{1}{4}}\right),
\]
and the lefthandside can be made arbitrarily small if $K$ is large enough,
and our claim follows from an adaptation
of the techniques employed in \cite{Cortissoz,Cortissoz2}. The following
lemma shows that the solution is becoming small in $\mathcal{PM}^2$-norm, 
modulo a tiny bad set.

\begin{lemma}
\label{longtermlemmagoodmodes}
Let $\xi\in R_{1,2}$ and $\xi \notin \mathcal{B}_0$,
let $u$ be a solution to (\ref{FourierNS0}) obtained from
the scheme (\ref{FourierNS2}).
Then, for $\left(\log\log K\right)^{\frac{1}{4}}\leq
t<3\left(\log\log K\right)^{\frac{1}{4}}$, the following estimate holds
\[
\left|\xi\right|^2\left|\hat{u}^l\left(\xi,t\right)\right|
\leq
2\left(\log\log K\right)^{\frac{1}{4}}\exp\left(-\left(\log\log K\right)^{\frac{1}{4}}\right)+
\frac{1}{K^{\frac{1}{16}}}.
\] 
\end{lemma} 
\begin{proof}
Notice that by Lemma \ref{importantbounds}, $E_n \leq \frac{1}{K^{\frac{1}{16}}}$ as long as $n\rho\leq 1$. 
This implies for $\xi\in R_{\omega,8}\setminus \mathcal{B}_0$ the
estimate 
\[
\left|\xi\right|^2\left|\hat{u}_{\rho}^l\left(\xi,t\right)\right|
\leq 2\left(\log\log K\right)^{\frac{1}{4}}\exp\left(-\left|\xi\right|^2 t\right)
+\frac{1}{K^{\frac{1}{16}}}\left(1-\exp\left(-\left|\xi\right|^2 t\right)\right),
\]
as long as $t<3\left(\log\log K\right)^{\frac{1}{4}}$, and the result follows.
\end{proof}

Since 
\[
2\left(\log\log K\right)^{\frac{1}{4}}\exp\left(-\left(\log\log K\right)^{\frac{1}{4}}\right)+
\frac{1}{K^{\frac{1}{16}}}\rightarrow 0 \quad \mbox{as}\quad K\rightarrow \infty,
\]
we can fix $\epsilon>0$ tiny and then take $K$ large enough so that 
\[
2\left(\log\log K\right)^{\frac{1}{4}}\exp\left(-\left(\log\log K\right)^{\frac{1}{4}}\right)+
\frac{1}{K^{\frac{1}{16}}}<\epsilon.
\]
On the other hand, by Lemma \ref{importantbounds}, $B_n\leq 10 \left(\log\log K\right)^{\frac{11}{4}}$ as long
as 
\[
t<3\left(\log\log K\right)^{\frac{1}{4}}.
\]
Then, the nonlinear term can be bounded, for any $\xi$ and any 
\[
t\in \left(\left(\log\log K\right)^{\frac{1}{4}},3\left(\log\log K\right)^{\frac{1}{4}}\right),
\]
and for $K$ large enough,
by a constant (independent of $\xi$ and $K$) times
\begin{equation}
\label{longtermnonlinear}
\epsilon^2+200\left(\log\log K\right)^{\frac{11}{2}}\theta \quad 
\mbox{(recall that $\theta=\sigma^{\frac{1}{3}}\leq \frac{1}{K^{\frac{1}{12}}}$)}.
\end{equation}
Now, by a similar argument as the one used in the proof of Lemma \ref{longtermlemmagoodmodes}, one can show 
that for $\epsilon>0$, there exists a $K$ large enough and a
\[ 
T_0\in \left(\left(\log\log K\right)^{\frac{1}{4}},3\left(\log\log K\right)^{\frac{1}{4}}\right)
\]
such that 
\[
\left\|u\left(T_0\right)\right\|_2<\epsilon
\]
for $\epsilon>0$ tiny. Indeed, to show this
it is enough to analyze the behavior of $\hat{u}^l\left(\xi,\cdot\right)$ if $\xi\in \mathcal{B}_0$. But
such a frequency number satisfies the following differential inequality on 
$\left(\left(\log\log K\right)^{\frac{1}{4}},3\left(\log\log K\right)^{\frac{1}{4}}\right)$,
\[
\frac{d}{dt}\hat{u}^l\left(\xi,t\right)\leq -\left|\xi\right|^2 \hat{u}^l\left(\xi,t\right)
+c\left(\epsilon^2+8\left(\log\log K\right)^{\frac{3}{4}}\theta\right),
\]
with $\left|\xi\right|^2\left|u^l\left(\xi,\left(\log\log K\right)^{\frac{1}{4}}\right)\right|
<10\left(\log\log K\right)^{\frac{11}{4}}$. Recalling that 
$\theta\leq \frac{1}{K^{\frac{1}{12}}}$, our claim follows easily.
Hence, by the results in \cite{CK}, for $t > T_0$, problem (\ref{FourierNS0}) admits a global regular solution; and by the
arguments given before, the solution we found up to time $T_0$ is smooth. Thus, we have shown,
\begin{theorem}
\label{lit}
Let $\gamma\in\Gamma$ be such that
the set of densities of
\[
\begin{array}{c}
\mathcal{B}_0=\\
\left\{
\xi\in R_{\omega,8}: 
\begin{array}{c}
\sum_{H\in \mathcal{P}} 
\left|\int_{q\in H}
M_{kjl}\left(\xi\right) \hat{\psi}_{\gamma}^k\left(q\right)\hat{\psi}_{\gamma}^j\left(\xi-q\right)\,d^3q\right|>
 \frac{1}{K^{\frac{1}{4}}}
\end{array}
\right\}
\end{array}
\]
is less than $\frac{1}{K^{\frac{1}{4}}}$.
Then $\psi_{\gamma}$ generates a global regular solution of (\ref{FourierNS0}).
\end{theorem}

All that is left to show is that there exists functions taken from
the set $\mathcal{I}_{K}$ for which the hypothesis of Theorem \ref{lit} holds, and this is the purpose
of the following section. 

\section{Showing that there are good Initial Conditions.}
\label{probabilistic}

Consider the function
\[
\Psi:\,R_{\omega,8}\times \Gamma \longrightarrow \mathbf{C}^{K}
\] 
given by
\[
\Psi\left(\xi,\gamma\right)=\left(\int_{q\in H_s}
\sum_{k,j=1,2,3}M_{kjl}\left(\xi\right) \hat{\psi}^k_{\gamma}\left(\xi-q\right)
\hat{\psi}^j_{\gamma}\left(q\right)\,d^3q\right)_{s=1,2,\dots,K}.
\] 
This is clearly a measurable function.
Observe that 
for $\gamma\in \Gamma$ and $\psi_{\gamma}$ 
defined by (A)-(E) using the divergence-free condition, we can rewrite the nonlinear term as
\[
\sum_{k,j=1}^2\tilde{M}_{kjl}\left(\xi\right)\hat{\psi}_{\gamma}^k\left(\xi\right)\hat{\psi}_{\gamma}^j
=
\sum_{k,j=1}^3 M_{kjl}\left(\xi\right)\hat{\psi}_{\gamma}^k\left(\xi\right)\hat{\psi}_{\gamma}^j,
\]
i.e., the original sum in the nonlinear term runs from 1 to 3, whereas after rewriting, by expresing 
$\hat{\psi}^3_{\gamma}\left(\xi\right)$ in terms of $\hat{\psi}^1_{\gamma}$ and $\hat{\psi}^2_{\gamma}$, 
via (D)-(E), the
new sum runs from 1 to 2. 
And hence we can write,
 \[
\Psi\left(\xi,\gamma\right)=\left(\int_{q\in H_s}
\sum_{k,j=1,2}\tilde{M}_{kjl}\left(\xi\right) \hat{\psi}^k_{\gamma}\left(\xi-q\right)
\hat{\psi}^j_{\gamma}\left(q\right)\,d^3q\right)_{s=1,2,\dots,K}.
\]
Notice also, that $\tilde{M}_{kjl}\left(\xi\right)$ satisfies an estimate,
\[
\left|\tilde{M}_{kjl}\left(\xi\right)\right|\leq C_{\kappa}\quad \mbox{on} \quad R_{\omega,8}.
\]
This is implicitly used but not explicitly stated 
in the following calculations.
Also, for our probabilistic arguments, this second version of $\Psi$ is more convenient (just to make use of the independence
of the random variables involved). 
We are ready to show that, for $\xi$ fixed, with high probability the nonlinear term is small.
We warn the reader that some of the computations below are up to constants
that are uniformly bounded, so they do not affect the
order of magnitude of any of the quantities whose asymptotic behavior depends on $K$.

\begin{proposition}
\label{probabilisticestimate1}
There is a constant $\beta>0$ independent of $K$ and $\xi$ such that for all $\xi\in R_{\omega,8}$
\begin{eqnarray*}
&P\left[\gamma:\,\exists\,H\,\mbox{such that} \,
 \left|\int_{q\in H}M_{kjl}\left(\xi\right) 
\hat{\psi}_{\gamma}^k\left(\xi-q\right)\hat{\psi}_{\gamma}^j\left(q\right)\,d^3q\right|\geq 
\frac{\left(\log\log K\right)^{\frac{1}{2}}}{K^{2-\delta}}\right]&\\
&\leq \beta K\exp\left(-cK^{2\delta}\right).&
\end{eqnarray*}
\end{proposition}

\begin{proof}
Assume first that $\xi$ is such that a each subblock of $\xi-R_{1,2}$ 
intersects a subblock of $R_{1,2}$ then it overlaps exactly with that subblock of $R_{1,2}$, i.e.,
\[
\mbox{if} \quad \left(\xi-W_{s,p}\right)\cap W_{s',p'}\neq \emptyset \quad
\mbox{then}\quad \xi-W_{s,p}= W_{s',p'}.
\]
Notice that at most $O\left(1\right)$ subblocks of $R_{1,2}$ overlap with themselves
(by $h=O\left(g\right)$ we mean $h\leq C\cdot g$ where $C$ is a constant independent of $K$). Indeed, let $W_{s,p}$ be a subblock such that
\begin{equation}
\label{superimpose}
\xi-W_{s,p}=W_{s,p},
\end{equation}
where for a set $W$ the notation $a-W$ has the usual meaning, i.e.,
\[
a-W=\left\{a-w: w\in W\right\}.
\] 
Hence, if (\ref{superimpose}) holds,
for each $\xi'\in W_{s,p}$ there is a $\xi''\in W_{s,p}$ such that $\xi'+\xi''=\xi$. But 
$\xi''=\xi'+O\left(\frac{1}{K}\right)$, and hence $\xi'=\frac{\xi}{2}+O\left(\frac{1}{K}\right)$, i.e., all elements
of the subblock must be within distance $\frac{1}{K}$ from $\frac{\xi}{2}$, and clearly only $O\left(1\right)$
subblocks satisfy this.

Write $W_{s',p'}=\xi-W_{s,p}$. We compute
\begin{eqnarray*}
\int_{q\in H_s} \hat{\psi}_{\gamma}^k\left(\xi-q\right)\hat{\psi}_{\gamma}^j\left(q\right)\,d^3 q
&=& \sum_{p} r_{s,p}r_{s',p'} \frac{\Theta_{K,s,p}^k\Theta_{K,s',p'}^j}{K^3}.
\end{eqnarray*}
Observe also that we always have
\[
\left(\xi-W_{s,p}\right)\cap \left(-W_{s,p}\right)=\emptyset, 
\]
as the diameter of the subblocks is $O\left(\frac{1}{K}\right)$, and we are assuming $\left|\xi\right|\geq \frac{1}{2}K^{-\frac{1}{8}}$.

Now, by our previous remarks, $r_{s,p}r_{s',p'}$ behave as a family of independent random variables
that take the value $\pm 1$ with probability $\frac{1}{2}$ 
(except for $O\left(1\right)$ of them), so if we consider the random variable 
\[
\sum_{p}r_{u,p}^kr_{u',p'}^j\frac{\Theta_{K,s,p}^k\Theta_{K,s',p'}^j}{K^3},
\]
recalling that there are $K^2$ subblocks per block, using the fact 
that $\left|\Theta_{K,s,p}^j\right|=\left(\log\log K\right)^{\frac{1}{2}}$, 
and since 
\[
\left|z\right|\geq a \quad \mbox{implies that either} \left|Re\left(z\right)\right|\geq\frac{a}{2}
\quad \mbox{or} \quad \left|Im\left(z\right)\right|\geq\frac{a}{2},
\]
Chernoff's bound (or Hoeffding's inequality, see Theorem 2 in \cite{Hoeff}) produces the estimate
\[
P\left[\gamma:\,\left|\int_{q\in H_s} \hat{\psi}_{\gamma}^k\left(\xi-q\right)\hat{\psi}_{\gamma}^j\left(q\right)
\,d^3 q\right|\geq \frac{\left(\log\log K\right)^{\frac{1}{2}}}{K^{2-\delta}}\right]
\leq \beta\exp\left(-cK^{2\delta}\right),
\]
where $\beta>0$ is a constant independent of $K$. Since there are $K$ blocks, the proposition follows.

In the case there is no exact overlapping of subblocks
(i.e. some subblocks intersect
but not exactly overlap), the random variable to be considered is
\[
\sum_{p, p'}r_{s,p}^kr_{s',p'}^j\Theta^k_{K,s,p}\Theta^j_{K,s',p'}\mu\left(W_{s,p}\cap W_{s',p'}\right)
\]
where $\mu\left(A\right)$ represents, as before, the Lebesgue measure of the set $A$. As
each subblock intersects at most $8$ other subblocks, and again 
at most
$O\left(1\right)$ subblocks will intersect themselves after reflection and translation, and 
$$\mu\left(W_{s,p}\cap W_{s',p'}\right)=O\left(\frac{1}{K^3}\right),$$ 
the previous argument applies and the proposition follows.
\end{proof}

Let us employ the notation $\left|\Psi\left(\xi,\gamma\right)\right|\geq \eta$ to say that there is an $s$ such that
\[
\left|\int_{q\in H_s}M_{kjl}\left(\xi\right) \hat{\psi}_{\gamma}^k\left(\xi-q\right)
\hat{\psi}_{\gamma}^l\left(q\right)\,d^3q\right|\geq \eta.
\]
From now to the end of this paper $\beta>0$ is a constant independent of $K$ that may change from line to line.
Now, we have from the previous proposition that
\[
\frac{1}{\mu\left(R_{\omega,8}\right)}
\int_{R_{\omega,8}}\int_{\left\{\gamma:\,\left|\Psi\left(\xi,\gamma\right)\right|
\geq \frac{\left(\log\log K\right)^{\frac{1}{4}}}{K^{2-\delta}}\right\}}\,dP\left(\gamma\right)d\mu\left(\xi\right)
\leq \beta K\exp\left(-cK^{2\delta}\right).
\]
Hence, from Fubini's Theorem we obtain,
\[
\frac{1}{\mu\left(R_{\omega,8}\right)}
\int_{\left\{\left(\xi,\gamma\right):\, 
\left|\Psi\left(\xi,\gamma\right)\right|\geq \frac{\left(\log\log K\right)^{\frac{1}{2}}}{K^{2-\delta}}\right\}} 
\,d\mu\left(\xi\right)\leq
\beta K\exp\left(-cK^{2\delta}\right).
\]

From the previous estimates the following can be deduced. Let 
\[
\mathcal{U}=\left\{\gamma:\, \mu\left\{\xi:\, 
\left|\Psi\left(\xi,\gamma\right)\right|\geq \frac{\left(\log\log K\right)^{\frac{1}{2}}}{K^{2-\delta}}\right\}
\geq \frac{1}{K^{\delta}}\right\},
\]
and write
\[
\mathcal{A}=
\left\{\left(\xi,r\right):\, \left|\Psi\left(\xi,\gamma\right)\right|
\geq \frac{\left(\log\log K\right)^{\frac{1}{2}}}{K^{2-\delta}}\right\}.
\]
Then $P\left[\mathcal{U}\right]$ is less than $\beta K^{1+\delta}\exp\left(-cK^{2\delta}\right)$, for otherwise,
by Cavalieri's principle, 
\[
\frac{1}{\mu\left(R_{\omega,8}\right)}
\int_{\mathcal{A}}\, d\mu\left(\xi\right)\times dP\left(\gamma\right)
\geq
\int_{\mathcal{U}}
\frac{1}{K^{\delta}}\, dP\left(\gamma\right)
> \beta K\exp\left(-cK^{2\delta}\right),
\]
which is a contradiction.

Summarizing we have shown the following
\begin{proposition}
Let
\[
\begin{array}{c}
\mathcal{B}_{0,\delta}=\\
\left\{
\xi\in R_{\omega,8}: 
\begin{array}{c}
\sum_{H\in \mathcal{P}} 
\left|\int_{q\in H}
M_{kjl}\left(\xi\right) \hat{\psi}_{\gamma}^k\left(q\right)\hat{\psi}_{\gamma}^j\left(\xi-q\right)\,d^3q\right|\geq
 \frac{\left(\log\log K\right)^{\frac{1}{2}}}{K^{1-\delta}}
\end{array}
\right\}.
\end{array}
\]
There are constants $\beta>0$ and $c>0$ independent of $K$, such that
\[
P\left[\gamma: \,\frac{\mu\left(\mathcal{B}_{0,\delta}\cap R_{\omega,8}\right)}{\mu\left(R_{\omega,8}\right)}\leq K^{-\delta}\right]
\geq 
1-\beta K^{1+\delta}\exp\left(-c K^{2\delta}\right),
\] 
i.e., with high probability $\mathcal{B}_{0,\delta}$ has small density.
\end{proposition}

By taking $\delta = \frac{11}{16}$, and taking into account Theorem \ref{lit}, the main result of this paper is proved.
Indeed, all is left to show is that when we take $\delta=\frac{11}{16}$ in the previous proposition, then the
family of densities of $\mathcal{B}_0:=\mathcal{B}_{0,1}$ is less than $\frac{1}{K^{\frac{1}{4}}}$.
Taking into account that $\mu\left(R_{\omega,2\omega}\right)=(const)K^{-\frac{3}{8}}$, this follows from the estimates
\[
\delta_{j}=\frac{\mu\left(\mathcal{B}_0\cap R_{2^j\omega,2^{j+1}\omega}\right)}{\mu\left(R_{2^j\omega,2^{j+1}\omega}\right)}
\leq \frac{\mu\left(\mathcal{B}_0\right)}{\mu\left(R_{\omega,2\omega}\right)}
\leq (const)\frac{1}{K^{\frac{11}{16}}}\times K^{\frac{3}{8}}<\frac{1}{K^{\frac{1}{4}}},
\]
when $K$ is large enough.

\section{On certain generalities}

From the proof it can be seen that we can take as the support of the Fourier transform any
compact set $R$ which is symmetric with respect to the origin (i.e. $R=-R$). On the other hand, we can take
a partition of $R\cap \mathcal{H}_{\kappa}\cap\mathbb{H}_{\geq 0}$ in blocks and subblocks and then
extend it to $R\cap \mathcal{H}_{\kappa}$, just as we did in the case of $R_{1,2}$. There
are restrictions on what type of sets blocks and subblocks can be:

-Blocks must be Lebesgue measurable disjoint sets, as well as subblocks, and the diameter of the blocks 
must be of the order of $K^{-\frac{1}{3}}$.  

-Subblocks must satisfy a condition akin to a finite intersection property. There must be a constant $C$
independent of $K$ such that given $\xi$ then at most $C$ subblocks $W$ intersect with $\xi-W$
(i.e., at most $C$ subblocks intersect with themselves after reflection and translation), and also,
given any subblock $W$ and $\xi\in \mathbb{R}^3$, $\xi-W$ intersects at most $C$ subblocks $W'$ (here $W\neq W'$).

Having defined the partition of the compact set $R$, we choose our random variables,
\[
r^j_{s,p}:\,\Gamma\longrightarrow \left[-1,1\right].
\]
This random variables should be independent and have zero mean -{\it it is not required of them to be identically distributed}.
Again we choose complex numbers $\left|\Theta^j_{K,s,p}\right|=\left(\log\log\right)^{\frac{1}{4}}$, $j=1,2$, and define
the initial data $\psi_{\gamma}$ just as we did before. 

The conditions imposed on blocks and 
subblocks guarantees that the mean and the variance of the random variable
\[
\sum_{s,p} r^k_{s,p}r^j_{s',p'}\Theta_{K,s,p}^k \Theta_{K,s',p'}^j
\]
are at most $O\left(\frac{\left(\log\log K\right)^{\frac{1}{2}}}{K^3}\right)$
and $O\left(\frac{\log\log K}{K^4}\right)$ respectively.  

At this point, we can recur to Chebyshev's inequality, which gives, with the same notation as in
proposition \ref{probabilisticestimate1}, 
the estimate
\begin{eqnarray*}
&P\left[\gamma:\,\exists\,H\,\mbox{such that} \,
 \left|\int_{q\in H}M_{kjl}\left(\xi\right) 
\hat{\psi}_{\gamma}^k\left(\xi-q\right)\hat{\psi}_{\gamma}^j\left(q\right)\,d^3q\right|\geq 
\frac{\left(\log\log K\right)^{\frac{1}{2}}}{K^{\frac{3}{2}-\delta}}\right]&\\
&\leq \frac{\beta}{K^{2\delta}}.&
\end{eqnarray*}
From here, another version Theorem \ref{maintheorem1} follows directly from
a similar argument to the given above. To be more specific, it can be shown,
for the case discussed in this section, that
 \[
P\left[\gamma:\, \psi_{\gamma}\,\, \mbox{generates a regular global solution to
(\ref{FourierNS0})}\right]\geq 1-\frac{c}{K}.
\]
where $c>0$ is a constant independent of $K$.




\end{document}